\documentclass[12pt]{amsart}
\usepackage{amssymb,latexsym}

\numberwithin{equation}{section} 

\newtheorem{theorem}{Theorem}[section]
\newtheorem{proposition}[theorem]{Proposition}

\newtheorem{corollary}[theorem]{Corollary}
\newtheorem{lemma}[theorem]{Lemma}

\theoremstyle{definition}

\newtheorem{remark}[theorem]{Remark}

\newtheorem{definition}[theorem]{Definition}

\keywords{cross-ratio, quasideterminant, quasi-Pl\"ucker 
coordinate}

\begin{document}

\title[Noncommutative Cross-Ratios]
{Noncommutative Cross-Ratios}
\author[Vladimir Retakh]{Vladimir Retakh}
\email{vretakh@math.rutgers.edu}
\address{\noindent Department of Mathematics, Rutgers University,
Piscataway, New Jersey 08854, USA}

\maketitle

\section{Introduction}

The goal of this note is to present a definition and to discuss
basic properties of cross-ratios over (noncommutative) division rings
or skew-fields. We present the noncommutative cross-ratios as products
of quasi-Pl\"ucker coordinates introduced in \cite{GR2} (see also
\cite{GGRW}). Actually, noncommutative cross-ratios were already mentioned
in a remark in \cite{GR2}. I decided to return
to the subject after my colleague Feng Luo explained to me the importance
of cross-ratios in modern geometry (see, for example, 
\cite{LM, L1, L2}).

I have to thank Feng Luo for valuable discussions. This work was partially supported by
Simons Collaborative Grant.

\section{Quasi-Pl\"ucker coordinates}

We recall here only the theory of quasi-Pl\"ucker coordinates for $2\times n$-matrices over a noncommutative division ring
$\mathcal F$. For general $k\times n$-matrices the theory is presented in \cite {GR2, GGRW}. 
Recall (see \cite {GR, GR1} and subsequent papers) that for a matrix 
$\begin{pmatrix} a_{1k}&a_{1i}\\ a_{2k}&a_{2i}\end{pmatrix}$ one can define four quasideterminants
provided the corresponding elements are invertible:
$$
\begin{vmatrix} \boxed{a_{1k}}&a_{1i}\\ a_{2k}&a_{2i}\end{vmatrix}=a_{1k}-a_{1i}a_{2i}^{-1}a_{2k}, \ \
\begin{vmatrix} a_{1k}&\boxed{a_{1i}}\\ a_{2k}&a_{2i}\end{vmatrix}=a_{1i}-a_{1k}a_{2k}^{-1}a_{2i},
$$
$$
\begin{vmatrix} a_{1k}&a_{1i}\\ \boxed{a_{2k}}&a_{2i}\end{vmatrix}=a_{2k}-a_{2i}a_{1i}^{-1}a_{1k},\ \ 
\begin{vmatrix} a_{1k}&a_{1i}\\ a_{2k}&\boxed{a_{2i}}\end{vmatrix}=a_{2i}-a_{2k}a_{1k}^{-1}a_{1i}.
$$

Let $A=\begin{pmatrix} a_{11}&a_{12}&\dots&a_{1n}\\a_{21}&a_{22}&\dots&a_{2n}\end{pmatrix}$ be a matrix
over $\mathcal F$.

\begin{lemma} Let $i\neq k$. Then 
$$
\begin{vmatrix} a_{1k}&\boxed{a_{1i}}\\ a_{2k}&a_{2i}\end{vmatrix}^{-1}   
\begin{vmatrix} a_{1k}&\boxed{a_{1j}}\\ a_{2k}&a_{2j}\end{vmatrix}=
\begin{vmatrix} a_{1k}&a_{1i}\\ a_{2k}&\boxed{a_{2i}}\end{vmatrix}^{-1}
\begin{vmatrix} a_{1k}&a_{1j}\\ a_{2k}&\boxed{a_{2j}}\end{vmatrix}^{-1}
$$
if the corresponding expressions are defined.     
\end{lemma}

Note that in the formula the boxed elements on the left and on the right must be on the same level. 

\begin{definition} We call the expressions
$$
q_{ij}^k(A)=\begin{vmatrix} a_{1k}&\boxed{a_{1i}}\\ a_{2k}&a_{2i}\end{vmatrix}^{-1}   
\begin{vmatrix} a_{1k}&\boxed{a_{1j}}\\ a_{2k}&a_{2j}\end{vmatrix}=
\begin{vmatrix} a_{1k}&a_{1i}\\ a_{2k}&\boxed{a_{2i}}\end{vmatrix}^{-1}
\begin{vmatrix} a_{1k}&a_{1j}\\ a_{2k}&\boxed{a_{2j}}\end{vmatrix}^{-1}
$$
the quasi-Pl\"ucker coordinates of matrix $A$. 
\end{definition}

Our terminology is justified by the following observation. Recall that
in the commutative case the expressions 
$$
p_{ik}(A)=\begin{vmatrix}a_{1i}&a_{1k}\\ a_{2i}&a_{2k}\end{vmatrix}=
a_{1i}a_{2k}-a_{1k}a_{2i}
$$
are the Pl\"ucker coordinates of $A$. One can see that in the commutative case
$$
q_{ij}^k(A)=\frac{p_{jk}(A)}{p_{ik}(A)},
$$ 
i.e. quasi-Pl\"ucker coordinates are ratios of Pl\"ucker coordinates.

Let us list properties of quasi-Pl\"ucker coordinates over (noncommutative)
division ring $\mathcal F$. We sometimes write $q_{ij}^k$ instead of $q_{ij}^k(A)$
where it cannot lead to a confusion.

\smallskip \noindent
1). Let $g$ be an invertible matrix over $\mathcal F$. Then
$$ q_{ij}^k(g\cdot A)=q_{ij}^k(A).$$

\smallskip \noindent
2) Let $\Lambda = \text {diag}\ (\lambda_1, \lambda _2,\dots, \lambda_n)$ be an
invertible diagonal matrix over $\mathcal F$. Then
$$ q_{ij}^k(A\cdot \Lambda)=\lambda_i^{-1}\cdot q_{ij}^k(A)\cdot \lambda_j.$$

\smallskip \noindent
3) If $j=k$ then $q_{ij}^k=0$; if $j=i$ then $q_{ij}^k=1$ (we always assume
$i\neq k$.)

\smallskip \noindent
4) $q_{ij}^k\cdot q_{j\ell}^k=q_{i\ell}^k\ $. In particular, $q_{ij}^kq_{ji}^k=1$.

\smallskip \noindent
5) ``Noncommutative skew-symmetry": For distinct $i,j,k$
$$
q_{ij}^k\cdot q_{jk}^i\cdot q_{ki}^j=-1.
$$
One can also rewrite this formula as $q_{ij}kq_{jk}^i=-q_{ik}^j$.

\smallskip \noindent
6) ``Noncommutative Pl\"ucker identity": For distinct $i,j,k,\ell$
$$
q_{ij}^k q_{ji}^{\ell} + q_{i\ell}^k q_{\ell i}^j=1.
$$

One can easily check two last formulas in the commutative case. In fact,
$$
q_{ij}^k\cdot q_{jk}^i\cdot q_{ki}^j=
\frac{p_{jk}p_{ki}p_{ij}}{p_{ik}p_{ji}p{kj}}=-1
$$
because Pl\"ucker coordinates are skew-symmetric: $p_{ij}=-p_{ji}$ for
any $i,j$.

Also, assuming that $i<j<k<\ell$
$$
q_{ij}^k q_{ji}^{\ell} + q_{i\ell}^k q_{\ell i}^j=
\frac{p_{jk}p_{i\ell}}{p_{ik}p_{j\ell}}+
\frac{p_{\ell k}p_{ij}}{p_{ik}p_{\ell j}}.
$$

Because $\frac{p_{\ell k}}{p_{\ell j}}=\frac {p_{k\ell}}{p_{j\ell}}$,
the last expression equals to
$$
\frac{p_{jk}p_{i\ell}}{p_{ik}p_{j\ell}}+
\frac{p_{k\ell }p_{ij}}{p_{ik}p{j\ell }}=
\frac {p_{ij}p_{k\ell}+p_{i\ell}p_{jk}}{p_{ik}p_{j\ell}}=1
$$
due to the celebrated Pl\"ucker identity
$$
p_{ij}p_{k\ell} - p_{ik}p_{j\ell} +p_{i\ell}p_{jk}=0.
$$

\begin{remark} We presented here a theory of the {\it left}
quasi-Pl\"ucker coordinates for $2$ by $n$ matrices where $n>2$.
A theory of the {\it right} quasi-Pl\"ucker coordinates for $n$ by $2$
or, more generally, for $n$ by $k$ matrices where $n>k$ can be found
in \cite{GR2, GGRW}.
\end{remark}

\section{Definition and basic properties of cross-ratios}

We define cross-ratios over (noncommutative) division ring $\mathcal F$
by imitating the definition of calssical cross-ratios in homogeneous coordinates,
namely, if the four points are represented in homogeneous coordinates by vectors 
$a,b,c,d$ such that $c=a+b$ and $d=ka+b$, then their cross-ratio is $k$.

Let
$$
x=\begin{pmatrix} x_1\\ x_2\\ \end{pmatrix},\ \
y=\begin{pmatrix} y_1\\ y_2\\ \end{pmatrix},\ \
z=\begin{pmatrix} z_1\\ z_2\\ \end{pmatrix},\ \
t=\begin{pmatrix} t_1\\ t_2 \end{pmatrix}
$$
be vectors in $\mathcal F^2$. We define the cross-ratio $\kappa=\kappa (x,y,z,t)$
by equations
$$
\begin{cases}
t=x\alpha +y\beta\\
z=x\alpha \gamma +y\beta \gamma \cdot \kappa
\end{cases}
$$
where $\alpha, \beta, \gamma, \kappa \in \mathcal F$.

Consider matrix
$$
\begin{pmatrix} 
x_1&y_1&z_1&t_1\\
x_2&y_2&z_2&t_2
\end{pmatrix}
$$
and index its columns by $x,y,z,t$.

\begin{theorem}
$$
\kappa (x,y,z,t)=q_{zt}^y\cdot q_{tz}^x\ .
$$
\end{theorem}

\begin{proof} We assume that coordinates of vectors $x,y,z,t$ are not equal
to zero. Equation $t=x\alpha +y\beta$ leads to the system
$$
\begin{cases} t_1=x_1\alpha +y_1\beta \\
t_2=x_2\alpha + y_2\beta
\end{cases}.
$$

The system implies that
\begin{equation}
\alpha = \begin{vmatrix} y_1&\boxed{x_1}\\
y_2&x_2\end{vmatrix}^{-1}
\begin{vmatrix} y_1&\boxed{t_1}\\
y_2&t_2\end{vmatrix}, \ \ 
\beta = \begin{vmatrix} x_1&\boxed{y_1}\\
x_2&y_2\end{vmatrix}^{-1}
\begin{vmatrix} x_1&\boxed{t_1}\\
x_2&t_2\end{vmatrix}
\end{equation}
(we treat $\alpha, \beta$ as unknowns here.)

 Equation $z=x\alpha\gamma +y\beta\gamma\kappa$ leads to the system
$$
\begin{cases} z_1=x_1\alpha\gamma +y_1\beta\gamma\kappa \\
z_2=x_2\alpha\gamma + y_2\beta\gamma\kappa
\end{cases}.
$$

The system implies that
\begin{equation}
\alpha \gamma = \begin{vmatrix} y_1&\boxed{x_1}\\
y_2&x_2\end{vmatrix}^{-1}
\begin{vmatrix} y_1&\boxed{z_1}\\
y_2&t_2\end{vmatrix}, \ \ 
\beta\gamma\kappa = \begin{vmatrix} x_1&\boxed{y_1}\\
x_2&y_2\end{vmatrix}^{-1}
\begin{vmatrix} x_1&\boxed{z_1}\\
x_2&z_2\end{vmatrix}\ .
\end{equation}

Then formulas (3.1), (3.2) imply that
$$
\gamma = \begin{vmatrix} y_1&\boxed{t_1}\\
y_2&t_2\end{vmatrix}^{-1}
\begin{vmatrix} y_1&\boxed{z_1}\\
y_2&z_2\end{vmatrix}
$$
and
$$ 
\kappa (x,y,z,t)= \begin{vmatrix} y_1&\boxed{z_1}\\
y_2&z_2\end{vmatrix}^{-1}
\begin{vmatrix} y_1&\boxed{t_1}\\
y_2&t_2\end{vmatrix}\cdot
\begin{vmatrix} x_1&\boxed{t_1}\\
x_2&t_2\end{vmatrix}^{-1}
\begin{vmatrix} x_1&\boxed{t_1}\\
x_2&t_2\end{vmatrix}=q_{zt}^yq_{tz}^x\ .
$$

The theorem is proved.
\end{proof}

\begin{corollary} Let $x,y,z,t$ be vectors in $\mathcal F$,
$g$ be a $2$ by $2$ matrix over $\mathcal F$ and
$\lambda_i\in \mathcal F$, $i=1,2,3,4$. If matrix $g$
and elements $\lambda_i$ are invertible then
\begin{equation}
\kappa (gx\lambda_1, gy\lambda_2, gz\lambda_3, gt\lambda_4)=
\lambda_3^{-1}\kappa (x,y,z,t)\lambda_3\ .
\end{equation}
\end{corollary}

Also, as expected,  in the commutative case the right hand side of (3.3) equals to
$\kappa (x,y,z,t)$.

A proof immediately follows from the properties of quasi-Pl\"ucker coordinates. 

\begin {remark}
Note that the group $GL_2(\mathcal F)$ acts on vectors in $\mathcal F^2$
by multiplication from the left: $(g,x)\mapsto gx$,  and the group $\mathcal F^{\times}$ of
invertible elements in $\mathcal F$ by multiplication from the right:
$(\lambda, x)\mapsto x\lambda^{-1}$. It defines the action of
$GL_2(\mathcal F)\times T_4(\mathcal F)$ on 
$P_4=\mathcal F^2\times \mathcal F^2\times \mathcal F^2\times \mathcal F^2$ where
$T_4(\mathcal F)=(\mathcal F^{\times})^4$.
The cross-ratios are {\it relative invariants} of the action.
\end{remark}

The following theorem generalizes the main property of cross-ratios to the noncommutive case.

\begin{theorem} Let $\kappa (x,y,z,t)$ be defined and $\kappa (x,y,z,t)\neq 0,1$.
Then $4$-tuples $(x,y,z,t)$ and $(x',y',z',t')$ from
$P_4$ belong to the same orbit of $GL_2(\mathcal F)\times T_4(\mathcal F)$ if and only if
there exists $\mu \in \mathcal F^{\times}$ such that
\begin{equation}
\kappa (x,y,z,t)=\mu \cdot \kappa (x',y',z',t')\cdot \mu ^{-1}\ .
\end{equation}
\end{theorem}

\begin {proof} If $(x,y,z,t)$ and $(x'y,z,t')$ belong to the same orbit then such $\mu$ exists
because $\kappa (x,y,z,t)$ is a relative invariant of the action. 

Assume now that formula (3.4) is satisfied and that $\kappa=\kappa (x,y,z,t)\neq 0,1$. Note that
the matrix $g$ with columns $z,t$ is invertible. Otherwise, there exists $\alpha \in \mathcal F$
such that $t=z\alpha$. If $\alpha =0$ then $\kappa =0$. If $\alpha \neq 0$ then
$q_{zt}^x=\alpha$ and $q_{tz}^y=\alpha^{-1}$. Then $\kappa =1$ which contradicts the assumption
on $\kappa$. The matrix $g'$ with columns $z',t'$ is also invertible.

By applying $g^{-1}$ to matrix with columns $x,y,z,t$ and $(g')^{-1}$ to the matrix with
columns $x',y',z,t'$ we get matrices
$$
A=\begin{pmatrix} a_{11}&a_{12}&1&0\\
a_{21}&a_{22}&0&1\end{pmatrix}, 
\ \ 
B=\begin{pmatrix} b_{11}&b_{12}&1&0\\
b_{21}&b_{22}&0&1\end{pmatrix}
$$
where $a_{ij}, b_{ij}\neq 0$ for $i,j=1,2$.
Then
$$
\kappa (x,y,z,t)=\begin{vmatrix} a_{12}&\boxed{1}\\
a_{22}&0\end{vmatrix}^{-1}
\begin{vmatrix} a_{12}&\boxed{0}\\
a_{22}&1\end{vmatrix}\cdot
\begin{vmatrix} a_{11}&\boxed{0}\\
a_{21}&1\end{vmatrix}^{-1}
\begin{vmatrix} a_{11}&\boxed{1}\\
a_{21}&0\end{vmatrix}\ .
$$

It implies that
$$
\kappa (x,y,z,t)=1\cdot (-a_{12}a_{22}^{-1})\cdot (-a_{11}a_{21}^{-1})^{-1}\cdot 1=
a_{12}a_{22}^{-1}a_{21}a_{11}^{-1}\ .
$$

Similarly, $\kappa (x',y,z,t')=b_{12}b_{22}^{-1}b_{21}b_{11}^{-1}$ and we have
\begin{equation}
a_{12}a_{22}^{-1}a_{21}a_{11}^{-1}=\mu \cdot b_{12}b_{22}^{-1}b_{21}b_{11}^{-1}\cdot \mu^{-1}\ .
\end{equation}

To finish the proof, it is enough to find invertible elements $\lambda _i\in \mathcal F$,
$i=1,2,3,4$ such that
$$
A=diag (\lambda_3,\lambda_4)\cdot B\cdot diag (\lambda_1^{-1}, \lambda_2^{-1},\lambda_3^{-1},\lambda_4^{-1})\ .
$$

It leads to the system
$$
\begin{cases}
a_{11}\lambda_1=\lambda_3b_{11}, \ \ a_{12}\lambda_2=\lambda_3b_{12}\\
a_{21}\lambda_1=\lambda_4b_{21}, \ \ a_{22}\lambda_2=\lambda_4b_{22}
\end{cases}
$$

Set $\lambda_3=\mu$. Then 
$$
\lambda_1=a_{11}^{-1}\mu b_{11},\ \ 
\lambda_2=a_{12}^{-1}\mu b_{12},\ \
\lambda_4=a_{21}^{-1}\lambda_1b_{21}^{-1}=a_{21}a_{11}^{-1}\mu b_{11}
$$
and the system is consistent if
$$
a_{21}a_{11}^{-1}\mu b_{11}b_{21}^{-1}=a_{22}\lambda_2b_{22}^{-1}=a_{22}a_{12}^{-1}\mu b_{12}b_{22}^{-1}
$$
which leads to formula (3.5) and the proof is finished.
\end{proof}

The following corollary shows that the defined cross-ratios satisfy {\it cocycle conditions}
(see \cite{Lab}).
\begin{corollary} For vectors $x,y,z,t,w$ 

\begin{align*}
\kappa(x,y,z,t)=\kappa (w,y,z,t)\kappa(x,w,z,t)\\
\kappa (x,y,z,t)=1-\kappa(t,y,z,x)
\end{align*}

\end{corollary}

\begin{proof} By definition
$$
\kappa (w,y,z,t)\cdot \kappa(x,w,z,t)=
q_{zt}^yq_{tz}^w\cdot q_{zt}^wq_{tz}^x=q_{zt}^yq_{tz}^x=
\kappa(x,y,z,t)\ .
$$
the second formula follows directly from the noncommutative
Pl\"ucker identity.
\end{proof}

The proposition also implies

\begin{corollary} For vectors $x,x_1, x_2,\dots x_n, z,t\in \mathcal F^2$
one has $$\kappa (x,x,z,t)=1$$ and
$$
\kappa(x_{n-1},x_n,z,t)\kappa(x_{n-2}, x_{n-1},z,t)\dots \kappa(x_1,x_2,z,t)=
\kappa(x_1,x_n,z,t)
$$
\end{corollary}

\begin{proof} Note that $$\kappa(x,x,z,t)=q_{zt}^xq_{tz}^x=1$$ and
\begin{align*}
\kappa(x_{n-1},x_n,z,t)\kappa(x_{n-2}, x_{n-1},z,t)\dots \kappa(x_1,x_2,z,t)=\\
=q_{zt}^{x_n}q_{tz}^{x_{n-1}}q_{zt}^{x_{n-1}}q_{tz}^{x_{n-2}}\dots q_{zt}^{x_2}q_{tz}^{x_1}=
q_{zt}^{x_n}q_{tz}^{x_1}=\kappa(x_1,x_n,z,t)\ .
\end{align*}

\end{proof}.

\section{Noncommutative cross-ratios and permutations}

There are $24$ cross-ratios defined for vectors $x,y,z,t\in \mathcal F^2$.
They are related by the following formulas:

\begin{proposition} Let $x,y,z,t\in \mathcal F$. Then
\begin{align}
q_{tz}^x\kappa (x,y,z,t)q_{zt}^x=
q_{tz}^y\kappa (x,y,z,t)q_{zt}^y=\kappa (y,x,t,z);\\
q_{xz}^y\kappa (x,y,z,t)q_{zx}^y=
q_{xz}^t\kappa (x,y,z,t)q_{zx}^t=\kappa (z,t,x,y);\\
q_{yz}^x\kappa (x,y,z,t)q_{zy}^x=
q_{yz}^t\kappa (x,y,z,t)q_{zy}^t=\kappa (t,z,x,y);\\
\kappa (x,y,z,t)^{-1}=\kappa (y,x,z,t).
\end{align}
\end{proposition} 

Note again the effect of conjugation in the noncommutative case since
$q_{ij}^k$ and $q_{ji}^k$ are inverses to each other.

\begin{proof} For (4.1) one has
$$
q_{tz}^x\kappa (x,y,z,t)q_{zt}^x=
q_{tz}^x\cdot q_{zt}^y q_{tz}^x \cdot q_{zt}^x=q_{tz}^xq_{zt}^y=
\kappa (y,x,t,z).
$$

The second part of the formula can be proved in a similar way.

For (4.4)
$$
\kappa (x,y,z,t)^{-1}=(q_{zt}^y\cdot q_{tz}^x)^{-1}=q_{zt}^x\cdot q_{tz}^x=
\kappa (y,x,z,t).
$$

For (4.2)
$$
q_{tz}^y\kappa (x,y,z,t)q_{zt}^y=q_{xz}^y\cdot q_{zt}^yq_{tz}^x\cdot q_{zx}^y=
q_{xt}^yq_{tz}^xq_{zx}^y=
$$
$$
=q_{xt}^y(-q_{tx}^zq_{xz}^t)q_{zx}^y=-q_{xt}^yq_{tx}^z\cdot q_{xz}^tq_{zx}^y=
$$
$$
=-(1-q_{xz}^yq_{zx}^t)q_{xz}^tq_{zx}^y=
-q_{xz}^tq_{zx}^y +1 =q_{xy}^tq_{yx}^z=\kappa (z,t,x,y)\ .
$$

Also,
$$
q_{xz}^t\kappa (x,y,z,t)q_{zx}^t=q_{xz}^t\cdot q_{zt}^yq_{tz}^x\cdot q_{zx}^t=q_{xz}^tq_{zt}^y(-q_{tx}^z)=
$$
$$
-q_{xz}^t(q_{zx}^yq_{xt}^y)q_{tx}^z=
-q_{xz}^tq_{zx}^y(1-q_{xz}^yq_{zx}^t)=
$$
$$
=-q_{xz}^tq_{zx}^y+1=q_{xy}^tq_{yx}^t=\kappa (z,t,x,y)\ .
$$

One can also prove (4.3) in a similar way.
\end{proof}

By using Proposition 4.1 and the cocycle condition one can get
all 24 formulas for cross-ratios of $x,y,z,t$.

\end{document}